\documentclass[10pt,halfline,a4paper]{article}

\usepackage[utf8]{inputenc}
\usepackage[english]{babel}
\usepackage{amsmath}
\usepackage{amsthm}
\usepackage{amsfonts}
\usepackage{amssymb}
\usepackage[pdftex]{color,graphicx}
\usepackage[titletoc]{appendix}
\usepackage{mathrsfs}  
\usepackage{caption} 
\usepackage{tikz-cd}
\usepackage{mathptmx}

\newtheorem{lemme}{Lemma}
\newtheorem{prop}{Proposition}

\newtheorem{theorem}{Theorem}
\newtheorem{corollaire}{Corollary}
\newtheorem{remark}{Remark}

\newcommand\E{ \mathbf{E} }
\newcommand\ii{ \mathbf{i} }
\newcommand\jj{ \mathbf{j} }
\newcommand\R{ \mathbf{R} }
\newcommand\PP{ \mathbf{P} }

\newcommand\Var{ \mathbf{Var} }

\newcommand\T{ \mathrm{T} }

\newcommand\G{ \mathrm{G} }
\newcommand\V{ \mathrm{V} }
\newcommand\EE{ \mathrm{E} }

\title{Spectra of Wishart Matrices \\ with size-dependent entries.}
 \author{Nathan Noiry\footnote{nathan.noiry@parisnanterre.fr}}
\date{}

\begin{document}
\maketitle

\begin{abstract}
We prove the convergence of the empirical spectral measure of Wishart matrices with size-dependent entries and characterize the limiting law by its moments. We apply our result to the cases where the entries are Bernoulli variables with parameter $c/n$ or truncated heavy-tailed random variables. In both cases, when $c$ goes to infinity or when the truncation is small, the limiting spectrum is a perturbation of the Marchenko-Pastur distribution and we compute its leading term.
\end{abstract}

\section{Introduction}
Let $X_n$ be a real random matrix of size $n \times m$ with i.i.d. entries. We define the Wishart matrix  $W_n = \frac{1}{n} X_nX_n^T$, where $X_n^T$ is the transpose of $X_n$. The spectral measure of $W_n$ is the random probability law:
\[  \mu_{W_n} = \frac{1}{n} \sum\limits_{ \lambda \in \text{Spec}(W_n)} \delta_{\lambda},   \]
where $\text{Spec}(W_n)$ is the spectrum of $W_n$ and $\delta_{ \lambda}$ the Dirac at $\lambda$. Since $W_n$ is a positive symmetric matrix, its eigenvalues are nonnegative reals. The work of Marchenko and Pastur \cite{marchenko1967distribution} implies that, when the entries have variance equal to $1$ and finite moments of all order. Then, almost surely, $\mu_{W_n}$ weakly converges to a probability law $\mu_{ \alpha}$ as $n,m \rightarrow + \infty$ and $m/n \rightarrow \alpha >0$. The law $\mu_{ \alpha}$ is given by: 
\[  \mu_{ \alpha}( \mathrm{d}x) =  \frac{\sqrt{(b-x)(x-a)}}{ 2 \pi x} \mathrm{d}x + \mathbf{1}_{ \alpha <1} \left(1- \alpha \right) \delta_0( \mathrm{d}x),  \]
where $a = ( 1 - \sqrt{\alpha})^2$ and $b = (1 + \sqrt{ \alpha } )^2$.

The main issue of this paper is to let the law of the entries of $X_n$ depend on $n$. Informally, our first result (Theorem \ref{Theorem A.S. Convergence}) states that in that case, under some moment conditions, the measures $\mu_{W_n}$ converge weakly to a probability law which is characterized by its moments, for which we provide a formula. This is an analog for Wishart matrices of a result obtained by Zakharevich for Wigner matrices in \cite{zakharevich2006generalization}. The method here is based on a proof of the convergence of all the moments of the spectral measures $\mu_{W_n}$. The $k$-th limiting moment will write:
\[    \sum\limits_{a=1}^{k} \sum\limits_{l=1}^{a} \alpha^l \sum\limits_{ \substack{ \mathbf{b}=(b_1, \ldots, b_a)   \\ \substack{  b_1 \geq b_2 \geq \ldots \geq b_a \geq 2 \\ b_1 + b_2 + \cdots + b_a = 2k  } } } | \mathcal{W}_k(a,a+1,l,\mathbf{b})| \prod\limits_{i=1}^{a} A_{b_i}.    \]
The set $\mathcal{W}_k(a,a+1,l,\mathbf{b})$ is a combinatorial object linked with closed words on planar rooted trees and encodes the combinatorics of moments. We give a precise definition in Section \ref{Section Zakharevich}. Interestingly, the $A_i$'s coefficients, given by the formula \eqref{Assumption Limiting Moment}, are the only reminders of the laws of the entries of the matrices. 

The convergence of the spectral measure was already proved in \cite{benaych2012marchenko} by Benaych-Georges and Cabanal-Duvillard, using different arguments. See also Male in \cite{male2017limiting} for related work. However, the main advantage of our approach is the explicit formula we obtain for the moments, which is more amenable to analysis, as we will see in Sections \ref{Section Zakharevich} and \ref{Section the Bernoulli case}. 

In \cite{vengerovsky2014eigenvalue}, Vengerovsky treated the particular case of diluted matrices $X_n(i,j) = a(i,j) d_n(i,j)$ where the $a(i,j)$'s are i.i.d. centered random variables and the $d_n(i,j)$'s are i.d.d. with Bernoulli law of parameter $c/n$. He derived a formula for the limiting moments, in terms of combinatorial quantities that admit a recursive formula. 

In the second part of the paper, we will focus on this particular case and let the entries of $X_n$ be i.i.d. Bernoulli laws with parameter $c/n$. In this setting, the Wishart matrices can be easily linked with the adjacency matrix of a bipartite random graph which admits a limit for the local weak topology. This convergence can be used to prove the convergence of the resolvent of the bipartite graph and therefore of $\mu_{W_n}$ itself, as explained in \cite{bordenave2010resolvent} by Bordenave and Lelarge. This is the content of Theorem \ref{Theorem Bernoulli case}. The limiting spectral measure $\mu_{ \alpha,c}$ depends only on $\alpha$ and $c$ and converges to the law $\mu_{\alpha}$ as $c \rightarrow \infty$. In Theorem \ref{Theorem Asympt Dvpt Bernoulli case}, we describe how $\mu_{ \alpha,c}$ differs from its limit $\mu_{ \alpha}$ by giving an asymptotic expansion in $1/c$ of its moments. More precisely, we will obtain that, in the sense of moments convergence:
\[  c \big( \mu_{ \alpha ,c} - \mu_{ \alpha } \big)  \underset{ c \rightarrow + \infty}{ \longrightarrow } \mu_{ \alpha}^{(1)},     \]
where $\mu_{ \alpha}^{(1)}$ is a signed measure of total mass zero, see Theorem \ref{Theorem Asympt Dvpt Bernoulli case}. The proof, based on a more careful analysis of the moment formula obtained in \ref{Theorem A.S. Convergence}, is inspired by the computations made in \cite{enriquez2015spectra} by Enriquez and Ménard. A natural extension would be to prove that the convergence holds in the sense of weak convergence, but it should involved new techniques since the moments of a signed measure of total mass zero do not characterized it. See Figure \ref{fig6} for numerical simulations.

\begin{figure}[h]
\hspace{-0.7cm}
\includegraphics[scale=0.4]{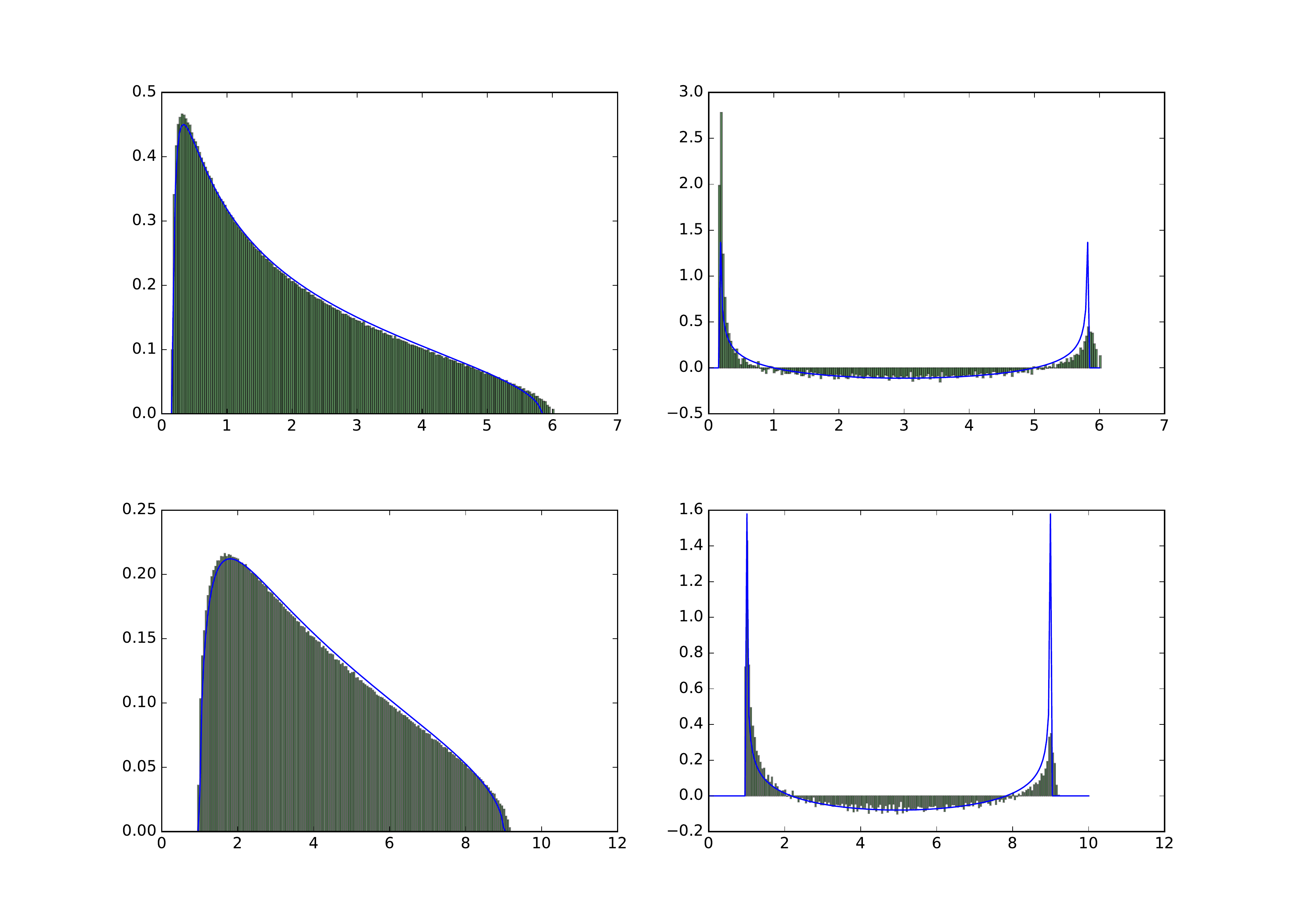}   
\captionof{figure}{Numerical simulations for the spectrum of $100$ Wishart matrices associated to random matrices of size $n \times \alpha n$ with i.i.d. entries with Bernoulli law of parameter $c/n$, with $c=20$ and $n=3000$. The theoritical densities of $\mu_{ \alpha }$ and $\mu_{\alpha}^{(1)}$ are drawn in blue. The top diagrams correspond to $\alpha =2$ whereas the bottom diagrams correspond to $\alpha=4$.}
\label{fig6}
\end{figure}
\bigskip
In the last part of this paper, we apply our results to heavy tailed random matrices. In that case, the entries of $X_n$ do not have finite moments of all order so that our main result does not apply. Instead, we truncate the entries at a constant $B>0$ times the largest $n$-th quantile of the corresponding law. By Theorem \ref{Theorem A.S. Convergence}, the spectral measures associated to the truncated random matrices converges to a deterministic probability law. The moments of this limiting law admit an asymptotic expansion involving the measures $\mu_{ \alpha}$ and $\mu_{ \alpha}^{(1)}$, as $B \rightarrow 0$. See Theorem \ref{Theorem Heavy Tailed}.

\section{A generalized Marchenko-Pastur theorem}\label{Section Zakharevich}
Let $ \mathcal{P} = \{P_n \}_{ n \geq 1}$ be a family of probability laws on $\R$ which have zero mean. For all $n \geq 1$, let $X_n = (X_n(i,j))_{1 \leq i,j \leq n}$ be a random $n \times m$ matrix with i.i.d. entries with law $P_n$. We will make the hypothesis that the ratio $m/n$ converges to a real $\alpha >0$ and that for all $k \geq 1$, the following limit exists and is finite:
\begin{equation}
A_k := \lim\limits_{n \rightarrow + \infty} \frac{M_k(P_n)}{n^{k/2-1} M_2(P_n)^{k/2}},
\label{Assumption Limiting Moment}
\end{equation}
where $M_k(P_n)$ is the $k$-th moment of $P_n$. Denote by $\mathcal{A}$ the sequence formed by the $A_k$'s. We are interested in the behavior of the spectral measures of the sequence of random matrices
\begin{equation*}
W_n := \frac{1}{n M_2(P_n)} X_n X_n^T.
\end{equation*} 

In order to properly state our first result, we need to introduce the notion of word on a labeled graph. A labeled graph is a graph $\G = (\V, \mathrm{E})$ together with a labeling of the vertices, that is a one-to-one application from $\V$ to $\{1, \ldots, |\V| \}$. A relabeling of a labeled graph is a new choice of bijection between $\V$ and $\{1, \ldots, |\V|\}$. Note that there are $|\V|!$ choices of labelings for a given graph $\G = (\V, \mathrm{E})$. A word of length $k \geq 1$ on a labeled graph $\G$ is a sequence of labels $i_1, i_2, \ldots, i_k$ such that $\{i_j, i_{j+1}\}$ is a pair of adjacent labels (that is the associated vertices are neighbours in $\G$) for all $1 \leq j \leq k-1$. A word of length $k$ is said to be closed if $i_1 = i_k$. Let $\ii = i_1, \ldots, i_k$ and $\ii' = i_1', \ldots, i_k'$ be two words of length $k$ on two labeled graphs $\G$ and $\G'$ having the same number of vertices. Then, $\ii$ and $\ii'$ are said to be equivalent if there exists a bijection $\sigma$ of $\{1, \ldots, |\V|\}$ such that $\sigma(i_j) = i_j'$ for all $1 \leq j \leq k$. In words, $\ii$ and $\ii'$ are equivalents if there exists a relabeling of a $\G$ such that the word associated to $\ii$ is exactly $\ii'$. One can check that this defines an equivalence relation on words on labeled graphs.

Recall that a planar rooted tree is a connected graph without loop embedded in the plane, with a distinguished vertex called the root. A vertex at odd (resp. even) distance from the root will be called an odd (resp. even) vertex. An edge with an odd (resp. even) origin vertex will be called an odd (resp. even) edge. 

\begin{theorem}\label{Theorem A.S. Convergence}
Suppose that for some $\gamma>0$, $A_k = O( \gamma^k )$ as $k\rightarrow +\infty$. Then there exists a probability law $\mu_{ \mathcal{A}, \alpha}$ depending only on $\mathcal{A}$ and $\alpha$, such that $\mu_{W_n}$ converges weakly to $\mu_{ \mathcal{A}, \alpha}$ in probability: for all $\varepsilon >0$ and all bounded continuous function $f: \R \rightarrow \R$,
\begin{equation*}
\PP \left( \left| \int_{\R} f \mathrm{d}\mu_{W_n} - \int_{\R} f \mathrm{d}\mu_{ \mathcal{A}, \alpha} \right| > \varepsilon \right)  \underset{ n \rightarrow +\infty }{ \longrightarrow } 0.
\end{equation*}
Moreover, the measure $\mu_{\mathcal{A}, \alpha}$ is characterized by its sequence of moments:
\begin{equation}
 M_k( \mu_{ \mathcal{A},\alpha}) =  \sum\limits_{a=1}^{k} \sum\limits_{l=1}^{a} \alpha^l \sum\limits_{ \substack{ \mathbf{b}=(b_1, \ldots, b_a)   \\ \substack{  b_1 \geq b_2 \geq \ldots \geq b_a \geq 2 \\ b_1 + b_2 + \cdots + b_a = 2k  } } } | \mathcal{W}_k(a,a+1,l,\mathbf{b})| \prod\limits_{i=1}^{a} A_{b_i},   
\label{Moments formula in the Theorem}
\end{equation}
where $\mathcal{W}_k(a,a+1,l,\mathbf{b})$ is a set of representatives of the equivalence classes of closed words on labeled rooted planar trees having ``$a$" edges, of which $l$ are odd edges, starting from the root and such that for all $1 \leq i \leq a$, one edge is browsed $b_i$ times.
\end{theorem}
\begin{remark}
The theorem can be thought as a universality result. Namely, if two sequences of probability law $P_n$ and $P_n'$ have the same asymptotic $\mathcal{A}$, the limiting spectral measures of $W_n$ and $W_n'$ are the same (in probability).
\end{remark}
\begin{corollaire}
If for all $k>2$, $A_k = 0$, the measures $\mu_{ \mathcal{A}, \alpha}$ and $\mu_{ \alpha }$ coincide. For example, this is the case when the laws $P_n$ are all equal.
\end{corollaire}
As the statement suggests, we are going to prove the result by the method of moments. Classically, we start with a computation of the average moments of $\mu_{W_n}$. For $k \geq 1$, we can write:
\begin{equation}
\E M_k(\mu_{W_n}) = \frac{1}{n^{k+1}M_2(P_n)^k} \sum\limits_{ \substack{1\leq i_1, \ldots, i_k \leq n \\ 1 \leq j_1, \ldots j_k \leq m} } \E[ X(i_1,j_1) X(i_2,j_1) \cdots X(i_k,j_k)X(i_1,j_k) ].
\label{Equation Average Moment}
\end{equation}
Denote $(\ii,\jj)$ the generic word $i_1j_1i_2 \ldots i_1j_k$ appearing in \eqref{Equation Average Moment}. We define the bipartite graph $\G=(\V,\EE)$ associated to the word $(\ii,\jj)$ by:
\[ \V = \{ (i_r, \ii), (j_r, \jj); \, 1 \leq r \leq k \} \quad \text{and} \quad \EE = \big\{ \{(i_r, \ii), (j_r, \jj)\}, \{ (i_{r+1}, \ii), (j_r, \jj) \}; \, 1 \leq r \leq k \big\},
\]
where we used the convention $k+1=1$. The abstract symbols $\ii$ and $\jj$ are needed to obtain a bipartite graph since the $i_r$'s and $j_r$'s can have common values (see Figure \ref{fig4} for illustration). We will refer to $\ii$ and $\jj$ letters. In words, the vertices of $\G$ are the letters of the word $( \ii, \jj)$ and two vertices are linked by an edge when they are consecutive in $(\ii,\jj)$. Denote by $s$ the number of vertices, $a$ the number of edges, $l$ the number of $\jj$-vertices and $\overline{l}$ the number of $\ii$-vertices in the word. Since $\G$ is connected, $s \leq a+1$. Moreover, since $P_n$ has zero mean, each edge must appear at least twice in the word to give a non-zero contribution in \eqref{Equation Average Moment}. As a consequence we obtain the bound $a \leq k$ because $i_1j_1 \ldots j_k$ possesses $2k$ edges counted with multiplicity. 

\begin{center}
\includegraphics[scale=0.65]{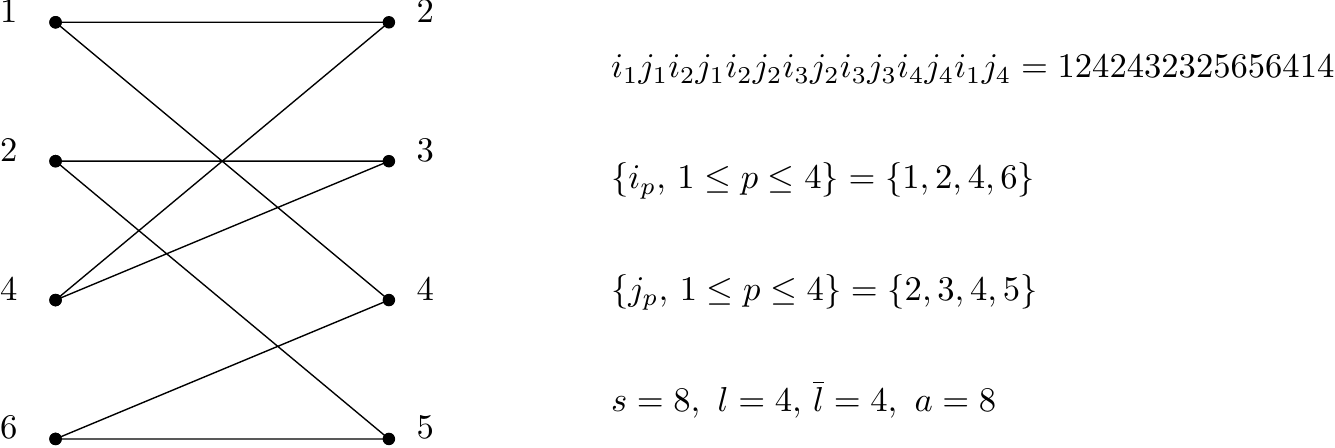} 
\captionof{figure}{Example of a word $(\ii,\jj)$ with its associated graph and quantities.}
\label{fig4}
\end{center}

Two words $(\ii,\jj)$ and $(\ii', \jj')$ are said equivalent if one can find a permutation $\sigma$ of $\{1, \ldots,n \}$ and another one $\tau$ of $\{1, \ldots, m \}$ such that
\[  \forall p \in \{1, \ldots k\}, \quad \sigma(i_p) = i_p' \, \, \, \text{and} \, \, \, \tau(j_p) = j_p'.  \]
One can check that this is an equivalence relation on the words appearing in \eqref{Equation Average Moment}. Note that $(\ii,\jj)$ has 
\[  C(s,l) = n(n-1) \cdots (n-l+s+1) \times m(m-1) \cdots (m-l+1) \sim \alpha^l n^s\]
equivalents. 
Fix $a \in \{1, \ldots, k\}$, $1 \leq s \leq a+1$ and $1 \leq l \leq a$. Let $\mathcal{B}_{a,k}$ be the set of $a$-tuples $\mathbf{b}=(b_1, \ldots, b_a)$ of integers such that
\begin{enumerate}
\item $b_1 \geq b_2 \geq \cdots \geq b_a \geq 2$;
\item $b_1+ \cdots +b_a = 2k$.
\end{enumerate}
For all $k \geq 1$ and $\mathbf{b} \in \mathcal{B}_{a,k}$, we introduce $\mathcal{W}_k(a,s,l,\mathbf{b})$ a set of representatives of the equivalence classes of words $(\ii,\jj)$ such that the associated graph has $a$ edges, $s$ vertices of which $l$ are $\jj$-vertices and such that for all $1 \leq i \leq a$ there is an edge in $\EE$ which has multiplicity $b_i$ in $(\ii,\jj)$. We can rewrite \eqref{Equation Average Moment} as:
\begin{equation}
\sum\limits_{a=1}^{k} \sum\limits_{s=1}^{a+1} \sum \limits_{l=1}^{s} \frac{C(s,l)}{n^{a+1}} \sum\limits_{b \in \mathcal{B}_{a,k}} \, \, \sum\limits_{ (\ii,\jj) \in \mathcal{W}(a,s,l,b)} \, \, \prod\limits_{1 \leq i \leq a} \frac{M_{b_i}(P_n)}{n^{b_i/2-1} M_2(P_n)^{b_i/2}}.
\label{Average Moment and Graphs} 
\end{equation}
From this equation we easily deduce the form of the limiting moments:
\begin{lemme}
An asymptotic contribution arises only if $s=a+1$ that is when the graph associated to $(\ii,\jj)$ is a tree. More precisely the limit of \eqref{Average Moment and Graphs} when $n \rightarrow +\infty$ is
\begin{equation}
M_k := \lim\limits_{n \rightarrow +\infty} \E M_k( \mu_{W_n} )= \sum\limits_{a=1}^{k} \sum\limits_{l=1}^{a} \alpha^l \sum\limits_{b \in \mathcal{B}_{a,k}} | \mathcal{W}_k(a,a+1,l,b)| \prod\limits_{i=1}^{a} A_{b_i}.
\label{What is M_k equal to}
\end{equation}
\end{lemme}
\begin{proof}[Proof]
Since $C(s,l)n^{-a-1} \sim \alpha^l n^{s-a-1}$ when $n \rightarrow +\infty$ we deduce that when $s<a+1$ the asymptotic contribution is zero. Hence a possible non-zero contribution arises only when $s=a+1$. The formula is a consequence of \eqref{Assumption Limiting Moment}.
\end{proof}
\begin{remark}\label{Remark on the parity}
The only non-zero contributions arise when $(\ii,\jj)$ is a walk on a tree that browses every edge and starts and finishes at the same vertex. Therefore each edge must be visited an even number of time: each $b_i$ in the tuple $\mathbf{b}$ is even. 
\end{remark}
\begin{remark}
The set $\mathcal{W}_k(a,a+1,l,\mathbf{b})$ is also the set of closed words on rooted planar trees having $a$ edges out of which $l$ are odd edges, starting from the root and such that for all $1 \leq i \leq a$, one edge is browsed $b_i$ times. Notice that the number of $\jj$-vertices is equal to the number of vertices in odd generations.
\end{remark}
In view of Theorem \ref{Theorem A.S. Convergence} we have to prove that $M_k( \mu_{W_n})$ concentrates around its mean. Since we are looking for a convergence in probability, it is sufficient to show that its variance vanishes when $n$ tends to infinity.
\begin{lemme}\label{Lemme Variance Sommable}
For all $k\geq 1$, $\Var( M_k(P_n) ) = O(n^{-1})$. In particular $M_k(P_n)$ converges to $M_k$ in probability. 
\end{lemme}
\begin{proof}[Proof]
Let $k \geq 1$. We can write $\E[  M_k(P_n)^2 ] - \E[  M_k(P_n)]^2$ as
\begin{equation}
\frac{1}{n^{2(k+1)}M_2(P_n)^{2k}} \sum\limits_{ (\ii,\jj),(\ii',\jj') }  \Big( \E[P(\ii,\jj)P(\ii',\jj')] - \E[ P(\ii,\jj) ] \E[ P(\ii',\jj')] \Big),
\label{Variance Write}
\end{equation}
where $P(\ii,\jj)$ is the product $X(i_1,j_1)X(i_2,j_1) \cdots X(i_1,j_k)$. We note $\G$ (resp. $\G'$) the graph associated to $(\ii,\jj)$ (resp. $(\ii',\jj')$), the corresponding quantities such as $s$ and $s'$ being defined as before. We also consider the graph $\G \cup \G'$ associated to $i_1 j_1 \ldots i_1 j_k i_1' j_1' \ldots i_1'j_k'$, and introduce $S$ its number of vertices, $A$ its number of edges and $L$ its number of $\jj$-vertices and $\jj'$-vertices. Note that if $\G$ and $\G'$ have no edge in common, then the contribution is zero by independence of $P(\ii,\jj)$ and $P(\ii',\jj')$. We can therefore restrict the sum to pairs of words $\big( (\ii,\jj), (\ii',\jj') \big)$ sharing at least an edge. In this case $\G \cup \G'$ is connected, hence $A \geq S-1$. Moreover, each edge must appear at least twice otherwise the contribution is zero since $P_n$ has mean zero. Therefore, the sum \eqref{Variance Write} rewrites
\begin{multline}
\frac{1}{n^{2(k+1)} M_2(P_n)^{2k}}\sum\limits_{S=1}^{2k+1} \, \, \sum\limits_{L=1}^{S} \, \, \sum \limits_{A=S-1}^{2k} C(S,L) \\
\times \sum\limits_{\mathbf{B} \in \mathcal{B}_{A,2k}} \, \, \, \, \sum\limits_{ (\ii,\jj),(\ii',\jj') \in \mathcal{W}_k(A,S,L,\mathbf{B})} \Big( \E[P(\ii,\jj)P(\ii',\jj')] - \E[ P(\ii,\jj) ] \E[ P(\ii',\jj')] \Big).
\label{Variance Write 2}
\end{multline}
Fix a generic couple $\big( (\ii,\jj),(\ii',\jj') \big)$. Let $e$ be an edge of $\G \cup \mathrm{G}'$. The corresponding $A$-tuple $\mathbf{B} \in \mathcal{B}_{A,2k}$ possesses a coefficient $B_i$ such that $e$ has multiplicity $B_i$. Note $b_i$ (resp. $b_i'$) the multiplicity of $e$ in $(\ii,\jj)$ (resp. $(\ii',\jj')$). We have the relation $b_i+b_i'=B_i$. The contribution of this generic couple in \eqref{Variance Write 2} is therefore
\begin{equation*}
\frac{C(S,L)}{n^{A+2}} \times \left(  \frac{\E[P(\ii,\jj)P(\ii',\jj')]}{\prod\limits_{1 \leq i \leq A} n^{B_i/2-1}M_2^{B_i/2}} \right. 
 \left. - \frac{\E[P(\ii,\jj)]}{ \prod\limits_{1 \leq i \leq A} n^{b_i/2-1}M_2^{b_i/2}  } \times  \frac{\E[ P(\ii',\jj')]}{\prod\limits_{1 \leq i \leq A} n^{b_i'/2-1}M_2^{b_i'/2} }   \right).
\end{equation*}
By assumption \eqref{Assumption Limiting Moment}, the absolute value of the difference inside the parentheses is bounded. This gives the conclusion since $C(S,L) \sim \alpha^L n^S$ and $S \leq A+1$.
\end{proof}
In order to obtain Theorem \ref{Theorem A.S. Convergence}, it remains to show that the sequence $\{M_k\}_{k \geq 1}$ entirely determines a probability law. To that aim, it is enough to prove that $M_k$ does not grow faster than $k^{ck}$ for some positive constant $c$. First, remark that
\[  |\mathcal{W}_k(a,a+1,l,\mathbf{b})|  \leq \frac{{(2k)}^k}{k+1} { 2k \choose k}.   \]
Indeed there are $\frac{1}{a+1} { 2a \choose a}$ rooted planar trees having $a$ edges. Moreover, two elements $(\ii,\jj)$ and $(\ii',\jj')$ in $\mathcal{W}(a,a+1,l,b)$ inducing the same tree differ only by the order in which each edge is browsed in the reading of $(\ii,\jj)$ (resp. $(\ii',\jj')$). For a fixed multiplicity $b_i$ and its associated edge $e$, there are at most ${ 2k \choose b_i}$ different possibilities to place the occurrences of $e$ because a word has $2k$ edges counted with multiplicity. Therefore the number of $(\ii,\jj) \in \mathcal{W}_k(a,s,l,\mathbf{b})$ associated to a fixed tree is bounded by 
\[   \prod\limits_{1 \leq i \leq a}  { 2k \choose b_i }  \sim \prod\limits_{1 \leq i \leq a} \frac{(2k)^{b_i}}{b_i!} \leq (2k)^{2k}.   \]
Using formula \eqref{What is M_k equal to} and the assumption $A_i = O(\gamma^i)$, we obtain the estimation
\[ M_k = O \left( \gamma^{k} \alpha^k k (2k)^{2k+1} \sum\limits_{a=1}^{k} | \mathcal{B}_{a,k} | \right).    \]
Finally the cardinality of $\mathcal{B}_{a,k}$ is bounded by the cardinality of the number of unsorted partitions of the integer $2k$. This last quantity is equal to
\[  \sum\limits_{i=1}^{2k} { 2k-1 \choose i-1}  = 2^{2k-1},   \]
where we summed over the number of partitions of $2k$ in $i$ parts. As a result $M_k = O(k^{ck})$ for some constant $c >0$. This concludes the proof of Theorem \ref{Theorem A.S. Convergence}.\\

The proof of the almost sure convergence of $\mu_{W_n}$ would require a concentration result analogous for instance to \cite[lemma 4.18]{bordenave2012around}. Rather than proving this kind of result, which would be a technical task, we present an alternative approach, specific to the case where the entries have Bernoulli law, borrowed from Bordenave and Lelarge's paper \cite{bordenave2010resolvent}. 

\section{The Bernoulli case}\label{Section the Bernoulli case}
In this section, we study the particular case where $P_n$ is the the centered Bernoulli law of parameter $c/n$, $c$ being a positive number, that is:
\[  P_n \left( - \frac{c}{n} \right) = 1 - \frac{c}{n} \quad \text{and} \quad P_n \left( 1 - \frac{c}{n} \right) =  \frac{c}{n}.  \]
In this case, since the second moment of $P_n$ verifies $M_2(P_n) = c/n + o(c/n)$ as $n \rightarrow \infty$, we set
\[  W_n = \frac{1}{c}X_n X_n^T \]
to simplify notations. We first give another proof for the convergence of the $\mu_{W_n}$, thanks to an interpretation of the hermitization of $X$ as the adjacency matrix of a random bipartite graph $\G_{n,m}$. This makes possible the use of the results of Bordenave and Lelarge in \cite{bordenave2010resolvent} after identifying the local limit of $\G_{n,m}$. In a second part, we give an asymptotic expansion in $1/c$ for the moments of the limiting spectral measure, inspired by Enriquez and Ménard (see \cite{enriquez2015spectra}).
\subsection{Another proof of the convergence}
To obtain the almost sure convergence of $\mu_{W_n}$, we will rather study the convergence of $W'_n = c^{-1} A_n A_n^T$, where $A_n$ is an $n \times m$ matrix having i.i.d. entries with (non-centered) Bernoulli law of parameter $c/n$. It is indeed sufficient because, denoting respectively $F$ and $F'$ the cumulative distribution functions of $\mu_{W_n}$ and $\mu_{W'_n}$, a consequence of Lidskii's inequalities is that:
\[  || F - F' ||_{ \infty} \leq \frac{ \text{rk}(X_n - A_n)}{n},  \]
where $\text{rk}$ is the rank operator. As announced before, we have the following theorem.
\begin{theorem}\label{Theorem Bernoulli case}
There exists a probability law $\mu_{ \alpha,c}$ depending only on $\alpha$ and $c$ such that, almost surely, $\mu_{W'_n}$ converges weakly to $\mu_{ \alpha,c}$. Hence, $\mu_{W_n}$ converges weakly to $\mu_{ \alpha,c}$. 
\end{theorem}
\begin{remark}
It can be proved that the set of atoms of $\mu_{ \alpha,c}$ is dense in $\mathbf{R}_+$. More precisely, it is the image by $x \mapsto x^2$ of the set of totally real algebraic integers, which coincides with the set of eigenvalues of finite trees as proved in \cite{salez2015every} by Salez. Besides, a consequence of the results of Bordenave, Sen and Virag in \cite{bordenave2013mean} is that $\mu_{ \alpha, c}$ possesses a continuous part if and only if $c>1$. 
\end{remark}

Define the hermitization of $A_n$ as the hermitian matrix:
\begin{equation}
H(A_n) = \begin{pmatrix}
   0 & A_n \\
   A_n^T & 0 
\end{pmatrix}.
\end{equation}
Remark that the spectrum of $H(A_n)$ is $\{ \pm  \sqrt{ \lambda_i(A_n A_n^T) }  \}_{1 \leq i \leq n}$. 
Let $f$ be the bijection of $\mathbf{R}_+$: $f(x)=x^2$. For a measure $\nu$ on $\mathbf{R}_+$ we define $\mathrm{Sym}( \nu )( \cdot ) = ( \nu( \cdot ) + \nu( - \cdot ) )/2$ the symmetrized version of $\nu$. Then 
\[ \mu_{ H(A_n) } = (\mathrm{Sym} \circ f_{*}) \mu_{A_n A_n^T}, \]
where $f_{*} \nu$ is the pushforward of a measure $\nu$ by $f$. Since $ \mathrm{Sym} \circ f_{*} $ defines a bijection between the measures which are supported on $\mathbf{R}_+$ and the symmetric measures on $\mathbf{R}$, it suffices to show the convergence of $\mu_{H(A_n)}$ to obtain Theorem \ref{Theorem Bernoulli case}. To avoid some unpleasant confusions, we will add an apostrophe to the asymptotic measures involved in the proof.

Now, $H(A_n)$ can be interpreted as the adjacency matrix of a random bipartite graph. Let $\mathrm{K}_{n,m}$ be the complete bipartite graph with $n$ and $m$ vertices of each color. The vertices of $\mathrm{K}_{n,m}$ will be denoted $1, \ldots, n+m$, two of them being linked by an edge if and only if one belongs to $\{1, \ldots, n \}$ and the other to $\{n+1, \ldots, n+m \}$. Perform a Bernoulli percolation with parameter $c/n$ on $K_{n,m}$: keep (independently) each edge with probability $c/n$ and remove it with probability $1-c/n$. We denote by $\G_{n,m}$ the resulting random graph. The adjacency matrix of $\G_{n,m}$ has the same law as $H(A_n)$. In the setting of local convergence introduced by Benjamini and Schramm \cite{benjamini2011recurrence} and Aldous and Steele \cite{aldous2004objective}, $\G_{n,m}$ converges in law to a random tree $\T_{ \alpha, c}$ for the local topology. To give a precise statement, we give some definitions in what follows.

For any connected, locally finite graph $\G$ and any vertex $v \in \G$ we will note $(\G,v)$ the class of pointed graphs isomorphic to the graph $\G$ pointed in $v$. For any $r \geq 0$, $[\G,v]_r$ will denote the ball of radius $r$ around $v$ in $\G$ for the graph distance. This induces a topology (called the local topology) on the set $\mathcal{G}^*$ of pointed graphs (up to isomorphism) which are locally finite and connected, making it a separable and complete space. 

For all nonnegative real number $x$ let $\mathcal{P}(x)$ denote the Poisson law with parameter $x$. Let $\T_{ \alpha, c, 1 }$ be the random tree where each individual reproduces independently from each other and such that individuals of an even and an odd generation reproduce respectively according to the laws $\mathcal{P}(c)$ and $\mathcal{P}( \alpha c)$. Let $\T_{ \alpha, c, 2}$ be the random tree where each individual reproduces independently from each other and such that individuals of an even and an odd generation reproduce respectively according to the laws $\mathcal{P}( \alpha c)$ and $\mathcal{P}(c)$. Notice that $\T_{ \alpha, c, 2}$ has the same law as the random tree issued from a children of the root of $\T_{ \alpha, c, 1}$. 

\begin{center}
\includegraphics[scale=0.65]{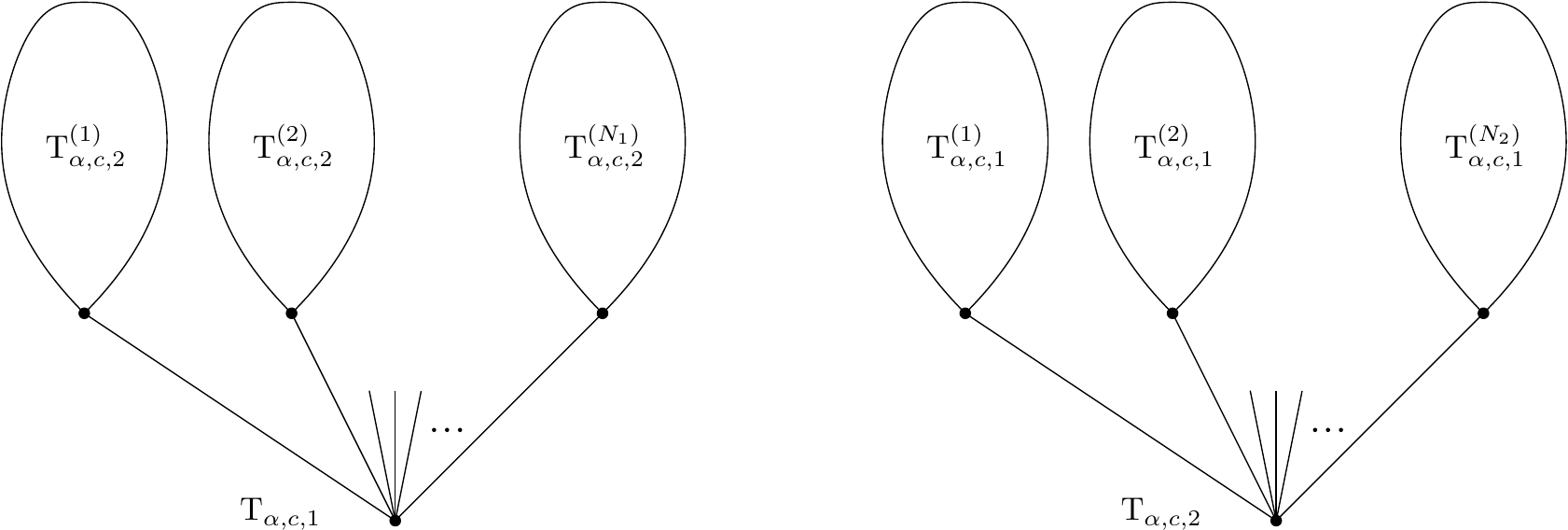} 
\captionof{figure}{The recursive relation between $\T_{ \alpha, c, 1}$ and $\T_{ \alpha, c, 2}$. Here $N_1$ and $N_2$ are independent random variables with law $\mathcal{P}(c)$ and $\mathcal{P}(\alpha c)$; and the $\T_{ \alpha, c, 1}^{(i)}$ (resp. the $\T_{ \alpha, c, 2}^{(i)}$) i.i.d. copies of $\T_{ \alpha, c, 1}$ (resp. $\T_{ \alpha, c, 2}$).}
\label{fig3}
\end{center}

Let $B$ be Bernoulli random variable of parameter $1 / ( \alpha + 1)$ independent of $\T_{ \alpha, c, 1}$ and $\T_{ \alpha, c, 2}$. We define $\mu_{ \alpha, c}'$ as the law of the random tree $\T_{ \alpha, c} := \mathbf{1}_{B=1} \T_{ \alpha, c, 1} + \mathbf{1}_{B=0} \T_{ \alpha, c, 2}$. Let $o$ be a uniformly distributed vertex on $\G_{n,m}$. We define the random probability measure on $\mathcal{G}^*$
\[ U_o(\G_{n,m}) := \delta_{ (\G_{n,m}(o),o)} = \frac{1}{n+m} \sum\limits_{i=1}^{n+m} \delta_{(\G_{n,m}(i),i)}.   \]
Integrating with respect to the randomness of $\G_{n,m}$ gives a new measure $\E[ U_o(\G_{n,m}) ]$ which is characterized by the relation $  \E[ U_o(\G_{n,m}) ](A) = \PP( (\G_{n,m}(o),o) \in A )$ for all measurable set $A \in \mathcal{B}(\mathcal{G}^*)$.

\begin{prop}\label{Local Convergence Theorem}
The deterministic probability measure $\E[U_o(G_{n,m})]$ converges weakly to $\mu_{ \alpha, c}'$. Moreover, if $o_1$ and $o_2$ are two independent copies of $o$, the product $ \E[U_{o_1}(G_{n,m})] \otimes \E[U_{o_2}(G_{n,m})]$ converges weakly to $\mu_{ \alpha, c}' \otimes \mu_{ \alpha, c}'$.
\end{prop}
\begin{proof}
The first part is a combinatorial argument that shows that $\PP( [\G_{n,m},1]_r \equiv t )$ converges to $\PP( [ \T_{\alpha,c,1}, \rho]_r \equiv t)$ as $n \rightarrow + \infty$ for all $r \geq 1$ and all rooted planar tree $t$ of depth $r$. For the second part, it suffices to remark that to independent uniform vertices $o_1$ and $o_2$ are almost surely at distance greater than $r$ as $n \rightarrow + \infty$, for any $r \geq 1$.
\end{proof}
Let us discuss the consequences of this proposition. It implies the validity of the main assumptions of the convergence theorem of Bordenave and Lelarge \cite[theorem 5]{bordenave2010resolvent}, relative to the empirical spectral measure of the adjacency matrix of large graphs having a local limit. What remains to check is the uniform integrability of the sequence of degrees $\{\deg_{ \G_{n,m} }(o)\}_{n \geq 1}$, which holds. The existence of a probability law $\mu_c'$ such that almost surely $\mu_{H(A_n)}$ converges weakly to $\mu_c'$ is then a direct application of a result of Bordenave and Lelarge \cite[theorem 5]{bordenave2010resolvent}. We also get a description of the Stieltjes transform of $\mu_c'$. Indeed \cite{bordenave2010resolvent} shows that there exists a unique pair of probability laws $(\mathcal{L}_1, \mathcal{L}_2)$ on the set of analytic functions on $\mathbf{C}_+ = \{ z \in \mathbf{C}: \, \, \mathrm{Im}(z) > 0 \}$ such that for all $z \in \mathbf{C}_+$:
\[ \left\{ \begin{array}{l}
        X_1(z) \overset{ (d) }{=} - \left(z + \sum\limits_{i=1}^{N_1} X_2^{(i)}(z) \right)^{-1} \\
        X_2(z) \overset{ (d) }{=} - \left(z + \sum\limits_{i=1}^{N_2} X_1^{(i)}(z) \right)^{-1},
    \end{array} 
    \right. 
\]
where $X_1$ and $X_2$ are independent random variables having laws $\mathcal{L}_1$ and $\mathcal{L}_2$, the $X_1^{(i)} $ and $X_2^{(i)}$ are i.i.d. copies of $X_1$ (resp. $X_2$), $N_1$ has law $\mathcal{P}(c)$ and $N_2$ has law $\mathcal{P}( \alpha c)$; each of these variables being independent. Then, the Stieltjes transform of $\mu_c'$ is given by:
\begin{equation}
 \forall z \in \mathbf{C}_+, \quad m_{\mu_c'}(z) := \int_{ \mathbf{R}} \frac{1}{x-z} \mathrm{d}\mu_c'(x) = \frac{1}{\alpha + 1} \E[ X_1(z) ] + \frac{\alpha}{\alpha + 1} \E[X_2(z)].
\end{equation}
This concludes the proof of theorem \ref{Theorem Bernoulli case} the limiting law being $\mu_c = (\mathrm{Sym} \circ f_*)^{-1}\mu_c'$.
\subsection{Asymptotic expansion of the moments}
Combining Theorem \ref{Theorem A.S. Convergence} with Theorem \ref{Theorem Bernoulli case}, we deduce that, almost surely, $\mu_{W_n}$ converges weakly to a probability law $\mu_{ \alpha, c}$ which is characterized by its sequence of moments. In that case, the asymptotic $\mathcal{A} = \{ A_k \}_{k \geq 1}$ of the laws $P_n$ which are Bernoulli laws of parameter $c/n$ is given by:
\begin{equation*}  A_k = \lim_{ n \rightarrow + \infty }   \frac{1}{n^{k/2-1}}  \frac{( 1 - c/n )^k (c/n) + (-c/n)^k (1-c/n) }{ [( 1 - c/n )^2 (c/n) + (-c/n)^2 (1-c/n)]^{k/2}  }  = c^{1 - k/2} \mathbf{1}_{ k > 1}. 
\end{equation*}
This leads to the following formula for the $k$-th moment of $\mu_{ \alpha,c}$:
\begin{equation}
M_k( \mu_{ \alpha,c}) =  \sum\limits_{a=1}^{k} \sum\limits_{l=1}^{a} \alpha^l \sum\limits_{   \substack{  b_1 \geq b_2 \geq \ldots \geq b_a \geq 2 \\ b_1 + b_2 + \cdots + b_a = 2k   } } | \mathcal{W}(a,a+1,l,b)| \prod\limits_{i=1}^{a} c^{1 - b_i/2}.
\label{Formula In The Bernoulli Case}  
\end{equation}
When $c \rightarrow + \infty$, we retrieve the moments of the Marchenko-Pastur law $\mu_{ \alpha }$. It is therefore natural to try to understand how $\mu_{ \alpha, c}$ differs from $\mu_{\alpha}$ when $c$ is large but finite. We give an answer to this question by giving an asymptotic expansion in $1/c$ of the moments of $\mu_{ \alpha,c}$. This is done by a more careful treatment of equation \eqref{Formula In The Bernoulli Case}, which is combinatorial in nature. The method is inspired by the paper \cite{enriquez2015spectra} of Enriquez and Ménard, where the authors treated the case of adjacency matrices of Erdös-Rényi graphs with parameter $c/n$. 
\begin{theorem}\label{Theorem Asympt Dvpt Bernoulli case}
There exists a signed measure $\mu_{\alpha}^{(1)}$ such that for all $k \geq 1$, as $c \rightarrow +\infty$:
\begin{equation}
M_k(\mu_{\alpha,c}) = M_k \left( \mu_{ \alpha} + \frac{1}{c} \mu_{\alpha}^{(1)} \right) + o \left( \frac{1}{c} \right).
\label{DVLPTAsymptoticTheorem}
\end{equation}
Moreover, the measure $\mu_{\alpha}^{(1)}$ has a total mass zero and the following density:
\begin{equation}
\frac{x^2-2x(\alpha+1)+(\alpha^2+1)}{2 \alpha \sqrt{(b-x)(x-a)}} \mathbf{1}_{(a,b)}.
\label{Density Of The Perturbation In the Theorem}
\end{equation}
\end{theorem}
\begin{proof}
Fix an integer $k \geq 1$. First, Remark \ref{Remark on the parity} ensures that all the $b_i$'s in \eqref{Formula In The Bernoulli Case} are even. Therefore, we can rewrite:
\begin{equation}
M_k( \mu_{ \alpha,c}) =  \sum\limits_{a=1}^{k} \sum\limits_{l=1}^{a} \alpha^l \sum\limits_{   \substack{  d_1 \geq d_2 \geq \ldots \geq d_a \geq 1 \\ d_1 + d_2 + \cdots + d_a = k   } } | \mathcal{W}_k(a,a+1,l,2d)| \prod\limits_{i=1}^{a} c^{1 - d_i}.
\label{Formula In The Bernoulli Case 2}  
\end{equation}
As $c \rightarrow + \infty$, nonvanishing terms correspond to the case where all the $d_i$'s are equal to $1$. This forces $a$ to be equal to $k$ and leads to:
\[  M_k( \mu_{ \alpha,c}) =  \sum\limits_{l=1}^{a} \alpha^l  | \mathcal{W}_k(k,k+1,l,(2, \ldots ,2))|  + o(1) \]
as $c \rightarrow +\infty$. 

Recall that $\mathcal{W}_k(k,k+1,l,(2, \ldots ,2))$ is a set of representatives of closed words starting at the root, of length $2k+1$ on labeled planar rooted trees having $k$ edges, $l$ of these being odd edges. This allows to write:
\[  M_k( \mu_{ \alpha,c}) =  \sum\limits_{ \T \in \mathcal{T}_k } \alpha^{l(\T)}  + o(1)  \]
as $c \rightarrow +\infty$, where $\mathcal{T}_k$ is the set of planar rooted trees having $k$ edges and $l(\T)$ the number of odd edges in a given tree $\T \in \mathcal{T}_k$. For convenience, we introduce the notations 
\[  a_k := \sum\limits_{ \T \in \mathcal{T}_k } \alpha^{l(\T)} \quad \text{and} \quad  b_k := \sum\limits_{ \T \in \mathcal{T}_k } \alpha^{ \overline{l}(\T)} , \]
where $\overline{l}(\T)$ is the number of even edges of a given tree $\mathrm{T} \in \mathcal{T}_k$.

It turns out that the $a_k$'s are the moments of $\mu_{ \alpha }$. To obtain the term of order $1/c$ we will need to compute the generating series of the $a_k$'s and $b_k$'s. 

Let $\T$ be a planar tree having $k+1$ edges. Let $\T_1$ be the tree induced by the first child of the root and $\T_2$ the connected component of the root after removing the edge between the root and its first child (see Figure \ref{fig5}).
{\begin{center}
\includegraphics[scale=0.5]{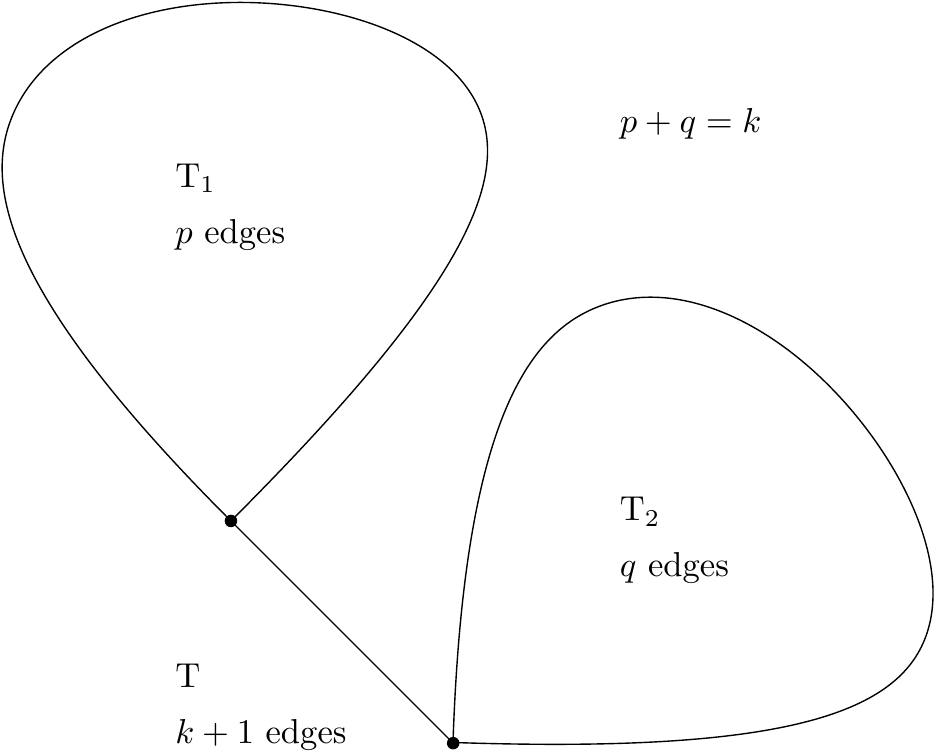} 
\captionof{figure}{Decomposition of a planar tree.}
\label{fig5}
\end{center}}
Denoting $p$ (resp. $q$) the number of edges of $\T_1$ (resp. $\T_2$), we have $p+q=k$. It is straightforward to obtain the relations $l(\T) = 1 + \overline{l}(\T_1) + l(\T_2)$ and $\overline{l}(\T) = l(\T_1) + \overline{l}(\T_2)$. Therefore 
\begin{equation}
\left\{ \begin{array}{l}
        a_{k+1} = \alpha \sum\limits_{p+q=k} a_pb_q \\
        b_{k+1} = \phantom{ \alpha} \sum\limits_{p+q=k} a_pb_q.
    \end{array} 
    \right. 
\end{equation} 
Denoting $A(z) = \sum_{k \geq 0} a_k z^k$ and $B(z) = \sum_{k \geq 0} b_k z^k$ the generating functions of the $a_k$'s and the $b_k$'s we obtain the functional relations:
\begin{equation}
\left\{ \begin{array}{l}
        A = 1 + \alpha z A B \\
        B = 1 + z A B .
    \end{array} 
    \right. 
\label{LinearRelationBetweenAandB}
\end{equation} 
It implies that $zA^2 + ( \alpha z - z -1)A +1 =0$. If we denote $S(z) := -z^{-1}A(z^{-1})$ the Stieltjes transform of the measure with moments $a_k$'s, then $S$ satisfies the equation:
\begin{equation} 
zS^2 - (\alpha - z - 1) S +1 =0. 
\label{Equation on the Stieltjes transform} 
\end{equation}
The function $S$ of the variable $z \in \mathbf{C}_+$ is the limit of the Stieltjes transform of the $\mu_{W_n}$ when $c \rightarrow +\infty$. The imaginary part of a Stieltjes transform is positive: this allows us to choose the right solution for equation \eqref{Equation on the Stieltjes transform}. For a complex $z$, if we denote $\sqrt{z}$ the square root having a positive imaginary part on the upper half plane:
\begin{equation}
S(z) = \frac{\alpha - z - 1  + \sqrt{(z-b)(z-a)}}{2z},
\label{Stieltjes formula S}
\end{equation}
where $a = (1 - \sqrt{ \alpha } )^2$ and $ b = (1 + \sqrt{ \alpha } )^2$. This is the Stieltjes transform of the Marchenko-Pastur law $\mu_{ \alpha}$, as announced.
\\

Let us compute the perturbation of order $1/c$. It arises when all the $d_i$'s are equal to $1$ except one which is equal to $2$ in \eqref{Formula In The Bernoulli Case 2}. This forces $a$ to be equal to $k-1$ and leads to the following expansion as $c \rightarrow +\infty$:
\[   M_k( \mu_{\alpha,c} ) = M_k( \mu_{ \alpha} ) + \frac{1}{c} \sum\limits_{ l = 1}^{k-1} \alpha^{l} | \mathcal{W}_k(k-1,k,l,(4,2, \ldots, 2))| + o \left( \frac{1}{c} \right).   \]
In that case $\mathcal{W}_k(k-1,k,l,(4,2, \ldots, 2)) $ is the set of equivalence classes of closed words of length $2k+1$ on labeled planar rooted tree having $k-1$ edges, starting at the root and such that each edge is browsed exactly two times except one which is browsed four times. Let us denote 
\[ a_k^{(1)} = \sum\limits_{ l = 1}^{k-1} \alpha^{l} | \mathcal{W}_k(k-1,k,l,(4,2, \ldots, 2))|, \] 
and
\[ b_k^{(1)} = \sum\limits_{ \overline{l} = 1}^{k-1} \alpha^{l} | \mathcal{W}_k(k-1,k,l,(4,2, \ldots, 2))|. \] 
The associated generating series will be denoted $A^{(1)}$ and $B^{(1)}$. Remark that by definition $a_0^{(1)} = a_1^{(1)} = b_0^{(1)}= b_1^{(1)} = 0$. We are going to obtain a recursion linking the four generating series $A,B,A^{(1)}$ and $B^{(1)}$. The idea is to use a first generation decomposition of the planar rooted tree on which the words are written, and then to distinguished whether or not the quadruple edge is an edge of this generation. For all $k \geq 1$, we use the partition
\[  \mathcal{W}_k(k-1,k,l,(4,2, \ldots, 2)) = \mathcal{W}_k^{(0)}(k-1,k,l,(4,2, \ldots, 2)) \bigsqcup \mathcal{W}_k^{(1)}(k-1,k,l,(4,2, \ldots, 2)),  \]
where $\mathcal{W}_k^{(0)}(k-1,k,l,(4,2, \ldots, 2))$ is the set of representative belonging to $\mathcal{W}_k(k-1,k,l,(4,2, \\ \ldots, 2))$ such that the quadruple edge is not a first generation edge, and $\mathcal{W}_k^{(1)}(k-1,k,l,(4,2, \ldots, 2))$ is the set of representatives belonging to $\mathcal{W}_k(k-1,k,l,(4,2, \ldots, 2))$ such that the quadruple edge is a first generation edge. The associated quantities will be denoted $a_k^{(1,0)},a_k^{(1,1)},A^{(1,0)},...$ For example:
\[  a_k^{(1,0)} = \sum\limits_{ l = 1}^{k-1} \alpha^{l} | \mathcal{W}_k^{(0)}(k-1,k,l,(4,2, \ldots, 2))|.  \]
A representative word $(\ii,\jj) \in \mathcal{W}_k(k-1,k,l,(4,2, \ldots, 2))$ can be written:
\[   (\ii,\jj) = i_1 \mathbf{S}_1 \zeta \xi \mathbf{S}_2 \xi \zeta \mathbf{S}_3 \zeta \xi \mathbf{S}_4 \xi \zeta \mathbf{S}_5 i_1,  \]
where:
\begin{enumerate}
\item $i_1 \mathbf{S}_1 \zeta \mathbf{S}_5 i_1$ is the contour of a planar tree having $p_1$ edges;
\item $\xi \mathbf{S}_2 \xi$ is the contour of a planar tree having $p_2$ edges;
\item $ \zeta \mathbf{S}_3 \zeta$ is the contour of a planar tree having $p_3$ edges;
\item $\xi \mathbf{S}_4 \xi$ is the contour of a planar tree having $p_4$ edges;
\item $\xi \mathbf{S}_2 \xi \mathbf{S}_4 \xi$ is the contour of a planar tree having $p_2+p_4$ edges.
\end{enumerate}
The above integers satisfy $p_1+p_2+p_3+p_4 = k-2$. See Figure \ref{fig6} for an illustration.
\begin{center}
\includegraphics[scale=0.75]{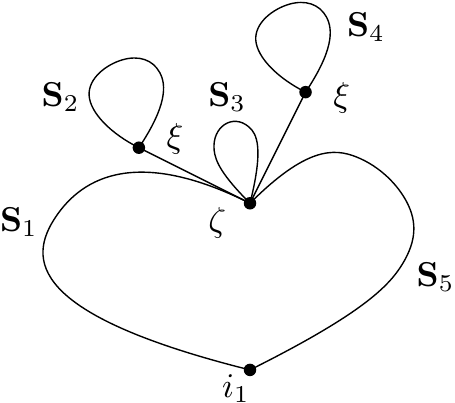} 
\captionof{figure}{The writing $(\ii,\jj)$ and its quadruple edge $\{ \zeta, \xi\}$.}
\label{fig6}
\end{center}
All of these conditions are sufficient to define a class of canonical representatives. Let $\T$ be the planar rooted tree on which a representative word $(\ii,\jj)$ is written. Denote $e_4$ the quadruple edge, $\T \setminus e_4$ the connected component of the root after removing $e_4$ and $\T^{e_4}$ the planar rooted tree formed by the descendants of $e_4$. Then, the above conditions ensures that $(\ii,\jj)$ is such that $\T \setminus e_4$ and $\T^{e_4}$ are respectively browsed in lexicographic order.

Let $(\ii,\jj) \in \mathcal{W}_k^{(0)}(k-1,k,l,(4,2, \ldots, 2))$. The underlying tree can have $p \in \{1, \ldots, k-2\}$ edges which are all browsed two times by $(\ii,\jj)$. One of the tree induced by the children of the root contains the quadruple edge, leading to $p$ different choices. On another side, if $(\ii,\jj) \in \mathcal{W}_k^{(1)}(k-1,k,l,(4,2, \ldots, 2))$ then the underlying tree can have $p \in \{1, \ldots, k-1\}$ edges out of which one is the quadruple edge. There are ${p+1 \choose 2}$ choices for the locations of the the visits of the quadruple edge. See Figure \ref{fig7} for an illustration.
\begin{center}
\includegraphics[scale=0.65]{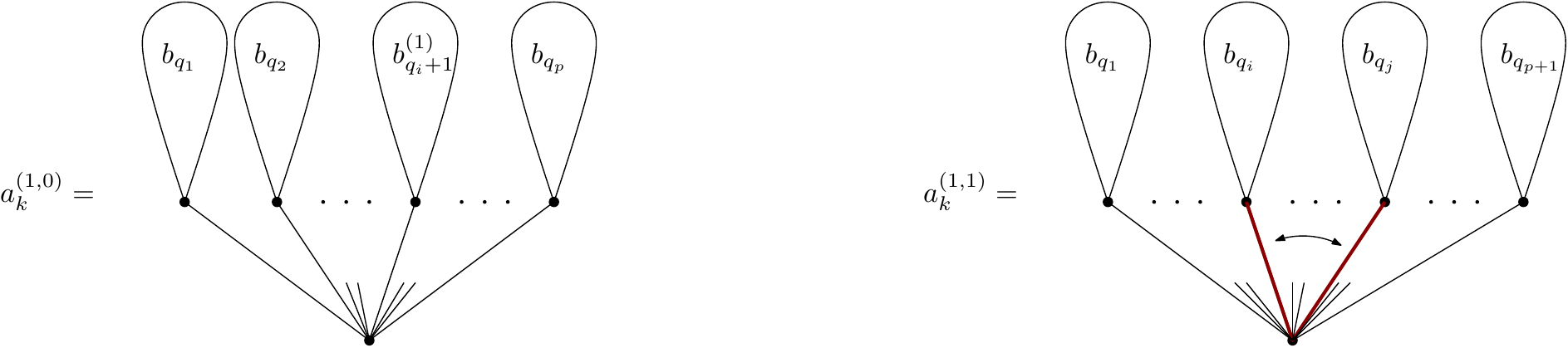} 
\captionof{figure}{First edge decomposition of a word respectively in $\mathcal{W}_k^{(0)}(k-1,k,l,(4,2, \ldots, 2))$ on the left and in $\mathcal{W}_k^{(1)}(k-1,k,l,(4,2, \ldots, 2))$ on the right, where the quadruple edge is in red.}
\label{fig7}
\end{center}

As a consequence, we get the following recursions:
\begin{equation*}
a_k^{(1,0)} = \sum\limits_{p=1}^{k-2} \alpha^p p \sum\limits_{q_1 + \cdots + q_p = k-p-1} b_{q_1+1}^{(1)} b_{q_2} \cdots b_{q_p}, \phantom{bblbllbb}
\label{Equation 1st gen out}
\end{equation*}
and
\begin{equation*}
a_k^{(1,1)} = \sum\limits_{p=1}^{k-2} \alpha^p {p+1 \choose 2} \sum\limits_{q_1 + \cdots + q_{p+1} = k-p-1} b_{q_1} b_{q_2} \cdots b_{q_{p+1}}.
\end{equation*}
This yields 
\[ A^{(1,0)} = \frac{\alpha z B^{(1)}}{(1- \alpha z B)^2} = \alpha z A^2 B^{(1)}
\]
and
\[ A^{(1,1)}= \frac{\alpha z^2 B^2}{(1- \alpha z B)^3} = \alpha z^2 A^3 B^2,
\]
where we used equation \eqref{LinearRelationBetweenAandB}. The same arguments and computations give $B^{(1,0)} = z A^{(1)} B^2$ and $B^{(1,1)} = z^2 A^2 B^3$, to finally obtain
\begin{equation}
\left\{ \begin{array}{l}
        A^{(1)} = \alpha z A^2 B^{(1)} + \alpha z^2 A^3 B^2 \\
        B^{(1)} = z A^{(1)} B^2 + z^2 A^2 B^3.
    \end{array} 
    \right. 
\end{equation} 
We deduce, using equation \eqref{LinearRelationBetweenAandB}, that $A^{(1)}$ is given by:
\begin{equation}
A^{(1)} = \frac{\alpha (zAB)^2}{1- \alpha (zAB)^2}(zA^2B+A) = \frac{AB}{1- \alpha (zAB)^2} \alpha (zAB)^2.
\label{premiere formule pour A1}
\end{equation}
To obtain a more explicit formula for $A^{(1)}$, one can compute $\alpha(zAB)^2$ using first that $B=(A+\alpha-1)/ \alpha$ and then that $zA^2 = (1-(\alpha-1)z)A-1$. After simplifications:
\begin{equation}  \alpha (zAB)^2 = \frac{(1- \alpha z - z)A +z -1}{\alpha z} = \frac{(\alpha^2+1)z^2-2z(\alpha-1)+1-(1-\alpha z - z) \sqrt{\delta}}{2 \alpha z^2}, 
\label{equation pour A}
\end{equation}
since $A = (2z)^{-1}(1-( \alpha-1)z-\sqrt{ \delta})$. Using that $\sqrt{\delta}=-2zA-(\alpha-1)z+1$, one can then check that $\sqrt{ \delta} AB = 1 - \alpha (zAB)^2$. From \eqref{equation pour A}, one can finally rewrite \eqref{premiere formule pour A1} as
\begin{equation*}
A^{(1)} = \frac{1}{\sqrt{\delta}} \frac{(\alpha^2+1)z^2-2z(\alpha+1)+1-(1-\alpha z - z) \sqrt{\delta}}{2 \alpha z^2}.
\end{equation*}
Therefore, the function $S^{(1)}(z) = -\frac{1}{z}A^{(1)}( \frac{1}{z})$ is given by
\begin{equation}
S^{(1)}(z) =   - \frac{z^2-2z( \alpha+1) +  (\alpha^2+1)}{2 \alpha \sqrt{(z-b)(z-a)}} + \frac{z-\alpha-1}{2 \alpha}.
\end{equation}
It corresponds to the Stieltjes transform of the measure $\mu_{\alpha}^{(1)}$ with density:
\[  -\frac{1}{\pi} \lim\limits_{ \varepsilon \rightarrow 0} \mathrm{Im} \big( S^{(1)}(x+ i \varepsilon) \big) = \frac{x^2-2x(\alpha+1)+(\alpha^2+1)}{2 \alpha \pi \sqrt{(b-x)(x-a)}} \mathbf{1}_{(a,b)}.   \]
This concludes the proof of Theorem \ref{Theorem Asympt Dvpt Bernoulli case}.
\end{proof}

The case $\alpha =1$, which corresponds to asymptotic square matrices $X_n$, should be emphasized. In this setting the density is
\[   \frac{1}{2 \pi} \frac{x^2-4x+2}{\sqrt{x(4-x)}} \mathbf{1}_{[0,4]},   \]
which corresponds to the pushforward by $x \mapsto x^2$ of the density obtained in \cite{enriquez2015spectra} by Enriquez and Ménard for the Wigner case, as expected.

\section{Heavy tailed random matrices}\label{Heavy tailed random matrices}
In this section we use Theorem \ref{Theorem A.S. Convergence} to study the spectral measure associated to heavy tailed Wishart matrices. For all $n\geq 1$, let $X_n$ be a random matrix of size $n \times m$ having i.i.d. entries with heavy tailed law $P$. As before, we suppose that the ratio $m/n$ converges to $\alpha>0$. We will consider the case where $P$ has density
\[   \frac{C(\beta)}{1+|x|^{\beta}},  \]
where $1< \beta <3$ and $C(\beta) = ( \int_{ \mathbf{R}} (1+|x|^{ \beta})^{-1} \mathrm{d}x )^{-1}$. Theorem $1.10$ of Belinschi, Dembo and Guionnet in \cite{belinschi2009spectral} ensures that, since $P$ is in the domain of attraction of a $(\beta-1)$-stable law, the spectral measure of 
\[ n^{ -\frac{2}{\beta-1} } X_nX_n^T\]
converges to a deterministic probability law $\mu_{ \alpha, \beta}$ depending only on $\alpha$ and $\beta$. 

To apply Theorem \ref{Theorem A.S. Convergence}, let us consider the truncated version of $X_n$. For all $n \geq 1$, let $P_n$ be the probability law given by
\[   P_n(\mathrm{d}x) = \frac{C}{1+|x|^{ \beta}} \mathbf{1}_{[-B n^{1/(\beta-1)},B n^{1/(\beta-1)}]} \mathrm{d}x + Z(B,\beta) \big( \delta_{-B n^{1/(\beta-1)}}( \mathrm{d}x) + \delta_{B n^{1/(\beta-1)}} ( \mathrm{d}x) \big),  \]
where $B>0$ and $Z(B,\beta) = 2C(\beta) \int_{B n^{1/(\beta-1)} }^{+ \infty} (1+|x|^{ \beta})^{-1} \mathrm{d}x$. In words, $P_n$ is the truncation of $P$ at $-B n^{1/(\beta-1)}$ and $B n^{1/(\beta-1)}$. We will denote $Y_n$ the random matrix of size $n \times m$ with i.i.d. entries having law $P_n$. Let us compute the asymptotic $\mathscr{A} = \{A_k\}_{ k \geq 2}$ of the sequence $\{P_n\}_{n \geq 1}$. For all $k \geq 1$, as $n$ tends to infinity:
\begin{align*}
\frac{M_k(P_n)}{n^{k/2-1} M_2(P_n)^{k/2}} &\sim n^{1-k/2} (2C)^{1-k/2} \, \frac{\int_1^{B n^{1/(\beta-1)}} x^{k-\beta} \mathrm{d}x + \frac{(B n^{1/(\beta-1)})^{k+1-\beta}}{1-\beta}}{( \int_1^{B n^{1/(\beta-1)}} x^{2-\beta} \mathrm{d}x + \frac{(B n^{1/(\beta-1)})^{3-\beta}}{1-\beta})^{k/2}} \\
                                          &\sim n^{1-k/2} (2C)^{1-k/2} \, \frac{\frac{(Bn^{1/(\beta-1)})^{k+1-\beta}}{k+1-\beta} + \frac{(B n^{1/(\beta-1)})^{k+1-\beta}}{1-\beta}}{( \frac{(B n^{1/(\beta-1)})^{3-\beta}}{3-\beta} + \frac{(B n^{1/(\beta-1)} )^{3-\beta}}{1-\beta})^{k/2}} \\
                                          &\sim n^{1-k/2} (2C)^{1-k/2} \, \frac{\frac{1}{k+1-\beta} + \frac{1}{1-\beta}}{ \frac{1}{3-\beta} + \frac{1}{1-\beta}} B^{1-\beta + \frac{k}{2}(\beta-1)} n^{ \frac{1}{\beta-1} \big( k+1-\beta + \frac{k}{2}(\beta-3) \big)}.
\end{align*}
We finally obtain:
\[ \frac{M_k(P_n)}{n^{k/2-1} M_2(P_n)^{k/2}} \sim (2C)^{1-k/2} \, \frac{\frac{1}{k+1-\beta} + \frac{1}{1-\beta}}{ \frac{1}{3-\beta} + \frac{1}{1-\beta}} B^{1-\beta + \frac{k}{2}(\beta-1)}.   \]

The quantity $n^{1/( \beta-1)}$ corresponds to the largest $n$-th quantile of $P$. Therefore, our choice of law $P_n$ can be interpreted as a truncation of the largest entries in each rows of $X_n$. If one had chosen an order of truncation smaller than $n^{1/(\beta-1)}$, the $A_k$'s would have been all equal to zero which corresponds to the Marchenko-Pastur regime, meaning that the truncation is too large and leads to a non-heavy tailed behavior. On the contrary, if one had chosen an order of truncation larger than $n^{1/(\beta-1)}$, the $A_k$'s would have been all infinite, meaning that the truncation is not large enough to apply Theorem \ref{Theorem A.S. Convergence}. In this spirit, the parameter $B>0$ can be seen as a finer adjustment of the truncation. 

Theorem \ref{Theorem A.S. Convergence} ensures that there exists a probability law $\mu_{\mathscr{A}, \alpha} = \mu_{\alpha,\beta,B}$ such that the spectral measures $\mu_n$ associated to the Wishart matrices $\frac{1}{nM_2(P_n)}Y_nY_n^T$ converges weakly in probability to $\mu_{\alpha,\beta,B}$. Using equation \eqref{Moments formula in the Theorem}, we obtain an asymptotic development of the moments of $\mu_{\alpha,\beta,B}$:
\begin{theorem}\label{Theorem Heavy Tailed}
For all $k \geq 1$, as $B \rightarrow 0$:
\[  M_k(\mu_{\alpha,\beta,B}) = M_k(  \mu_{ \alpha }) + B^{ \beta-1} \frac{1}{2C( \beta )} \cdot \frac{(3-\beta)^2}{(2-\beta)(5-\beta)} M_k \left( \mu_{ \alpha }^{(1)}   \right) + o\left( B^{\beta-1} \right).   \]
\end{theorem}
\begin{remark}
For simplicity we considered the explicit law $P(\mathrm{d}x) = \frac{C(\beta)}{1+|x|^{\beta}}$. However, using Karamata's estimates (Theorem 2, Section VIII.9 of \cite{feller1967introduction}) on truncated moments of regularly varying functions, one could have studied in a similar way the case when $P$ is in the domain of attraction of a $(\beta -1)$-stable law, for $1 < \beta <3$.
\end{remark}

\paragraph*{Acknowledgments.} The author would like to warmly thank his advisors Nathanaël Enriquez and Laurent Ménard for many helpful discussions and suggestions about this work.

\bibliographystyle{plain}

\bigskip
\noindent Nathan Noiry :\\
Laboratoire Modal'X, \\
UPL, Université Paris Nanterre,\\ 
F92000 Nanterre France

\end{document}